\documentclass[11pt]{amsart}
\usepackage{latexsym}
\usepackage{amsfonts,graphics}

\usepackage{amsmath,amsthm,amssymb}
\newtheorem{definition}{Definition}
\newtheorem{question}{Question}
\newtheorem{proposition}{Proposition}[section]
\newtheorem{theorem}[proposition]{Theorem}

\newtheorem{corollary}{Corollary}

\newtheorem{thm}{Theorem}[section]
\newtheorem{lemma}[thm]{Lemma}

\theoremstyle{definition}
\newtheorem{defn}[thm]{Definition}

\newcommand{\comment}[1]{}
\newcommand{\NP}{\mathcal{NP}}
\begin{document}

\title[Computational complexity and  the conjugacy problem]{Computational complexity  and  the conjugacy problem}

\date{\today}

\author[Alexei ~Miasnikov]{Alexei Miasnikov}

       \address{\tt Department of Mathematics, Stevens Institute of Technology,
Hoboken, NJ 07030, USA}
\email{amiasnik@stevens.edu   }

\author[Paul ~Schupp]{Paul Schupp}

\address{\tt Department of Mathematics, University of Illinois at
       Urbana-Champaign, 1409 West Green Street, Urbana, IL 61801, USA}
       \email{schupp@math.uiuc.edu}


\begin{abstract}
   The conjugacy problem for a finitely generated group $G$ is 
the  two-variable problem of deciding for an arbitrary pair $(u,v)$ of
elements of $G$, whether or not $u$ is conjugate to $v$ in $G$.  We construct 
examples of finitely generated, computably presented groups such that
for every  element $u_0$ of $G$, the problem of deciding if an
arbitrary element is conjugate to $u_0$ is decidable in quadratic time
but the worst-case complexity of the global conjugacy problem is arbitrary: it can be
any c.e. Turing degree , can exactly mirror the Time Hierarchy Theorem, or can be $\mathcal{NP}$-complete. 
Our groups also have the property that the conjugacy problem is generically linear time:
  that is,  there is a linear time partial algorithm for the conjugacy problem whose domain has density $1$,
so hard instances are very rare.  
We also consider the complexity relationship of the ``half-conjugacy'' problem to the conjugacy problem. 
In the last section we discuss the extreme opposite situation: groups with algorithmically finite conjugation.
\end{abstract}

\maketitle

\tableofcontents

\bigskip

 \section{Introduction}\label{intro}
\medskip

   The word problem is a special case of the conjugacy problem since an element  in a 
 group $G$ is equal to the identity if and only if it is conjugate 
to  the identity.  The  {\em individual conjugacy problem}, ICP($u_0$), for  a fixed element $u_0$ of a group $G$
is the problem of deciding whether or not  arbitrary elements $v$ of $G$ are conjugate to $u_0$.
The complexity of an individual conjugacy problem  depends
 on the fixed element $u_0$, and its complexity may be much less  than that of the 
conjugacy problem in $G$.  The word problem for semigroups resembles the conjugacy
problem for groups in that it is a two variable problem since an equation $u = v$  cannot generally
be reduced to an equation with one fixed side.  In 1965, J.C. Shepherdson \cite{Shepherdson} constructed
 finitely presented semigroups in which each individual problem $u_0 = v ?$,  with   $u_0$ fixed and  $v$ arbitary,   is  decidable
but the global word problem is undecidable of arbitrary c.e. degree.  Shortly thereafter Donald Collins \cite{Collins1} constructed 
finitely presented groups in which each individual conjugacy problem was decidable but the groups had undeciable conjuacy problem
of arbitrary c.e. degree.  Indeed,  both papers show that if one has a uniformly computable set $\boldsymbol{d}_i$ of c.e. degrees then
the complexity of each individual problem problem is bounded by  the join of a finite number of the $\boldsymbol{d}_i$ while the
global two-variable problem can be of any c.e. degree $\boldsymbol{d}$ greater than or equal to any of the $\boldsymbol{d}_i$.

  The constructions of Collins used HNN extensions.
  Although HNN extensions may generally have complicated conjugacy problems,  the 
time complexity of many individual conjugacy problems may be quite low. Miller  \cite{Miller2} 
 constructed a group $K$ which is an HNN extension of a free group of finite rank 
with finitely generated associated subgroups in which the conjugacy problem is undecidable. 
However, it was shown  in \cite{BMR3} that the individual conjugacy problem for a 
generic element $w_0 \in K$ is decidable in at most cubic time.   See
 \cite{BMR1,BMR2} for other examples of HNN extensions and free products with amalgamation
where the the complexity of the individual conjugacy problems of  many elements is low.

  In this article we construct finitely generated, computably presented groups where  \emph{all}  individual conjugacy problems are 
decidable in quadratic time but the global  conjugacy problem can have arbitrary worst-case complexity.
complexity. (Note in particular that  the word problem is decidable in quadratic time.)
We  discuss this    from the viewpoints  of the general theory of computability,
   the Time Hierarchy Theorem  and $\NP$-completeness.  There is also now a general awareness that many  decision problems are
generically easy and this phenomenon for decision problems in group theory was investigated in detail in \cite{KMSS}.  
Although having arbitrary worst-case complexity,
the conjugacy problem in  the groups we construct will  have linear time generic-case complexity.  This means that there is a  linear time partial algorithm  for the conjugacy problem whose domain has density $1$.

   The paper \cite{CDKV}  raised the  fascinating question of
the ``half-conjugacy problem''.   Suppose that we have a finitely generated group
$G$ with an algorithm which decides, given an arbitrary pair $(u,v)$ of 
elements of $G$, whether or not $u$ is conjugate to one of $v$ or $v^{-1}$:
Must $G$ have solvable conjugacy problem?  One supposes that the answer is 
``No'',  but the question seems  very subtle.  We do not answer the basic  question
but we show that for every  computable  function 
$f: \mathbb{N} \to \{0,1\}$,
there is a group $P$ with solvable conjugacy problem  in  which 
the half-conjugacy problem is decidable in quadratic time while  the conjugacy problem 
 has time complexity greater than $f$.  The group $P$   also satisfies  the above constraints on individual conjugacy
problems and generic-case complexity.

    Finally, in the last section we discuss the extreme opposite situation where 
complexity is the worst possible.  A finitely generated group  $G$ with a computably enumerable set of defining relators is
 \emph{algorithmically finite}  if every infinite computably enumerable subset has two distinct words which define elements
equal in $G$.  
Miasnikov and Osin \cite{MO} showed how to use the Golod-Shafarevich inequality to construct such groups.
We say that a finitely generated group $G$ has \emph{algorithmically finite conjugation} if $G$ has infinitely many conjugacy classes and every 
infinite computablely enumerable set of elements of $G$ must contain two elements which are conjugate.  We show that algorithmically
finite groups have algorithmically finite conjugation.

    We obtain  very precise control over complexity by using non-metric small cancellation theory so
 we first review the condition which which use.  This condition  ensures  that the structure  of conjugacy diagrams in the 
groups  we construct is  very simple, thus proving the desired results.
 We note that our groups require using an infinite number of relators since any  group with a finite presentation
satisfying our  condition  has  global conjugacy problem decidable in   quadratic time.
We then review  each  desired  complexity
condition and discuss the corresponding groups in  separate sections.  Howevery, there is   really only one basic construction
and the different cases require only small adjustments in the defining relators.

    We thank the referee for very helpful suggestions.

\section{Non-metic small cancellation theory}

   We construct the desired groups using  small cancellation theory.
For this we need  results developed in  Chapter V of Lyndon and Schupp~\cite{LS}
but we recall some essential definitions and details here.

\begin{definition} Let $ F = \langle X \rangle$ be a finitely generated free group.
A subset $R$ of $F$ is \emph{symmetrized} if all elements of $R$ are cyclically
reduced and, for  each $r \in R$,  all cyclic permutations of both $r$ and $r^{-1}$ are
also in $R$. We write $r \equiv bc$ if  $b$ is an initial segment of $r$ and $c$ consists
of the remaining letters of $r$, so $bc$ is reduced without cancellation.
 If $r_1 \ne r_2$ are distinct elements of $R$ with  $r_1 \equiv bc_1$ and
 $r_2 \equiv bc_2$ where $b$ is nonempty,  then $b$ is called a \emph{piece} relative to $R$.
\end{definition}

The basic non-metric small cancellation condition is

\[ \text{Condition} \ C(p):   \text{ ~ No element of } R \text{~ is a product of fewer than } p \text{~ pieces}. \] 

   The sets of defining relators which we construct will satisfy the condition $C(20)$.
Even though we will  not have a metric condition on the lengths of pieces, 
we need a good notion of \emph{reduction}.

\begin{definition} Fix a symmetrized subset $R$ of the free group $ F = \langle X \rangle$.
We assume that all generators are pieces.
If $w$ is any cyclically reduced word, consider the factorization $w \equiv b_1b_2 \dots b_l$
of $w$ into maximal pieces.  That is, each $b_i$ is a piece and, if $b_i$ is not a suffix of $w$
then $b_iy$ is not a piece where $y$ is the letter following $b_i$ in $w$.
The integer   $l$  is the \emph{piece length} of $w$ with respect to $R$.  
We denote the piece length of $w$ by $||w||_R$. 
\end{definition}

   The following lemma is immediate.

\begin{lemma} Let $R$ be a symmetrized subset of $F$.  If $w$  is a cyclically reduced word and $w'$
is a cyclic permutation of $w$ then  $||w||_R$ and  $||w'||_R$ differ by at most $1$.
\end{lemma}

 We need a notion of reduction tailored to  the sets of relators which we will use.  

\begin{defn}  A word $w$ \emph{contains an element of $R$ with at most $k$ pieces missing} 
if $w \equiv usv$ and there is an element $r \in R$ with $r \equiv st$ where $||t||_R \le k$.  
A word $w$ is \emph{weakly cyclically $R$-reduced}  if w is cyclically
reduced in the free group $F$,  $w$ does not begin and end with powers of the
same generator unless $w$ is simply a power of a generator, and no cyclic permutation of $w$  
contains a relator from $R$ with at most $7$ pieces missing.
\end{defn}

   The basic result of small cancellation theory is that if $G = \langle X; R \rangle$ where $R$
satisfies the condition $C(p)$ with $p \ge 6$ then every nontrivial word $w$ which is equal to
the identity in $G$ contains a subword which is  an element of $R$ with at most $3$ pieces missing.
(If $R$ satisfies the metric condition $C'(\frac{1}{6})$, so that the length of any piece is less
than $\frac{1}{6}$ of the length of any element of $R$ in which it occurs,  we then have Dehn's Algorithm.)

    \section{Groups with global conjugacy problem of arbitrary c.e. degree}
    
  A set $A$ of positive integers is \emph{computably enumerable}, written c.e., 
if there is a Turing machine $M$ which enumerates all the elements of $A$.
The basic relation between computably enumerable sets is that of
Turing reducibility.   A  set $A$ is \emph{Turing reducible} to a  set $B$,
 denoted $A \le_T B$, if there is an oracle Turing machine $M^B$ with an oracle for $B$
 which computes $A$.  Two sets $A$ and $B$ are \emph{Turing equivalent} if
 $A \le_T B$ and  $B \le_T A$.  Turing equivalence is indeed an equivalence relation
and  equivalence classes are called \emph{Turing degrees}.  A Turing degree 
$\boldsymbol{d}$ is \emph{c.e.} if it contains a c.e. set.  
(Not all sets in a nonzero c.e. degree are themselves
c.e.  since a set is always Turing equivalent to its complement.) The c.e. Turing degrees are
partially ordered by Turing reducibility and it is an important fact
that there infinitely many distinct c.e. Turing degrees.

  Traditionally, a \emph{recursive presentation} of a group
is a group presentation $\langle X; R \rangle$ where the set $X$ of generators
is finite and the set $R$ of defining relators is computably enumerable.
 A  \emph{computable presentation} of a group 
is a group presentation $\langle X; R \rangle$ where the set $X$ of generators
is finite and the set $R$ of defining relators is computable.   We use these terms
as distinguishing two different classes of presentations. If a group $G$ has
a recursive presentation it also has a computable presentation, but this requires
changing the presentation.

   It was shown in the early 1970's
that  the word and conjugacy problems for recursively  presented, indeed,  finitely presented,
groups mirror all possible  relations of Turing reducibility between c.e. Turing degrees:   
The c.e. Turing degrees $\boldsymbol{d}_1$  and $\boldsymbol{d}_2$ satisfy   
$\boldsymbol{d}_1 \le_T \boldsymbol{d}_2$
 if and only if there is a recursively   presented ( finitely presented) group $G$ with 
word problem of Turing degree $\boldsymbol{d}_1$  
and conjugacy problem of Turing degree $\boldsymbol{d}_2$.  The result for recursive 
presentations is in Miller \cite{Miller1} and the stronger version for finite presentations 
is due to Collins \cite{Collins2}.

  We will show

\begin{thm} \label{th:Turing} 
For every  c.e. Turing degree $\bf d$  there is computably presented group $G$
such that  for every fixed element $u_0$ of $G$, the problem of deciding if an
arbitrary element is conjugate to $u_0$ is decidable in quadratic time but the conjugacy
problem of $G$ has degree $\bf d$ but  linear time generic-case complexity.
\end{thm}

  Alexander Ol'shanskii and Mark Sapir \cite{OS} proved the  deep theorem that any 
finitely generated group with solvable conjugacy problem can be embedded in  a finitely
presented group with solvable conjugacy problem, establishing a direct conjugacy analog of the
Higman Embedding Theorem. Such a result depends on Ol'shanskii and Sapir's development 
of the theory of S-machines.  It seems plausible  that a very detailed analysis of the 
definitions and lemmas in their monograph applied to our groups would yield finitely 
presented groups where the complexity of all the individual conjugacy problems is bounded
 by a fixed tower of exponentials while the global conjugacy problem is undecidable, 
but this would need to be carefully verified.

    Since we want to discuss computational complexity, we need to be precise about how the lengths
of inputs are measured.  The standard length function for free groups is essentially a unary notation
since it requires reduced words to be completely written out.  Thus the length of $a^i$ is $i$.
We therefore use a  unary notation for positive integers, representing $n$ by a repetition of $n$ identical 
symbols.  It is a basic fact that if $A$ is any infinite c.e. set, there is a \emph{computable bijection}  
$f:\mathbb{N}^+ \to A$.  Given an infinite c.e. set $A$ not containing $0$,
fix such an $f$ and let $M$ be a Turing machine which, 
on input $i$ written in unary, computes $j = f(i)$ in unary. 
Let $t_i$ be the number of steps used by $M$ in computing $ f(i)$.  
Let $F = \langle a_1, \dots, a_{20},b_1, \dots, b_{20}, c_1, c_2, c_3, d_1,d_2,d_3 \rangle$.
The defining relators for the group $G$ for $A$ will be the symmetrized closure $R$ of the 
the  set $\{ r_i : i \in \mathbb{N}^+ \}$ of relators where 

\[   r_i = {a_1}^j \dots {a_{20}}^j {c_1}^i {d_1}^{t_i} {c_2}^i {d_2}^{t_i} {c_3}^i{d_3}^{t_i}   {b_{20}}^{-j} \dots {b_1}^{-j} {d_3}^{-t_i} {c_3}^{-i} {d_2}^{-t_i} {c_2}^{-i} {d_1}^{-t_i} {c_1}^{-i} , j = f(i)  \]
  
    Given the defining relators $R$, it will be immediate from small cancellation theory that 

\[  {a_1}^j \dots {a_{20}}^j   \thicksim  {b_{1}}^{j} \dots {b_{20}}^{j}  
   \iff  j \in A. \]
 
\noindent     and thus  the conjugacy problem in $G$ has  the same Turing degree as  $A$.
After discussing the geometry of conjugacy diagrams it will be clear that this is essentially
the only difficult case,  establishing the desired results.

     We first verify that $R$ satisfies the piece condition  $C(20)$.  The \emph{conjugating subwords}
of the relator $r_i$ are  the subword
$${c_1}^i {d_1}^{t_i} {c_2}^i {d_2}^{t_i} {c_3}^i{d_3}^{t_i}$$ 
and its inverse.  A conjugating part of a relator is of course a piece.
Note that for any full power of a generator $a_k$ or $b_k$ or $c_l$ occurring in an $r_i$, say  ${a_k}^j$,  
that power flanked  on \emph{both}  sides by  occurrences of the
neighboring generators, say $a_{k-1}{a_k}^j a_{k+1}$,  is \emph{not} a piece.  Since the function $f$ 
is one-to-one, ${a_k}^j$ and ${c_l}^{i}$ occur only in the relator $r_i$ where $ f(i) = j $.  
On the other hand,  the subword consisting of two successive powers, say ${a_k}^j  {a_{k+1}}^j$,
 is a piece because  the set $A$ is infinite and  there  are  $i' > i, j' > j$ with $f(i') =j' \in A$ 
so   ${a_k}^{j'}  {a_{k+1}}^{j'}$ occurs in the relator  $r_{i'}$.
  Since  the time $t_i$ 
required to compute $f(i)$ and  the time $t_{i'}$ required to compute $f(i')$  may be the same 
for different values of $i$ and $i'$, the three syllables at either end of the conjugating part,
say  ${c_3}^i {d_3}^{t_i} {b_{20}}^{-j}$,  may occur in different relators.

     The relators have  been  chosen so that the following lemma
holds.  Indeed, it will hold for the set of relators of all the  groups which  we construct.

\begin{lemma}[Reduction Lemma]
\label{l: reduction}
  There is an algorithm which, given an arbitrary word $w$,
 calculates  a weakly cyclically
$R$-reduced conjugate of $w$ in quadratic time.
\end{lemma}

\begin{proof}   First, in linear
time we calculate a cyclically reduced conjugate $w'$ in the free group
 which  does not begin and end with powers of the same generator,
unless it is simply  a power of a generator.  In the latter case, it
weakly  cyclically $R$-reduced.   The point of the form of the relators
is that a relator with at most $7$ pieces missing must contain
a \emph{critical subword} of either the form
 
\[    {a_{20}}^j {c_1}^i {d_1}^{t_i}{c_2}^i {d_2}^{t_i}{c_3}^i {d_3}^{t_i}{b_{20}}^{-j} \]
  
or the form

\[  {b_{1}}^{-j} {d_3}^{-t_i}{c_3}^{-i}  {d_2}^{-t_i}{c_2}^{-i}  {d_1}^{-t_i} {c_1}^{-i} {a_{1}}^{j}\]
  or their inverses.  This is because there are $9$ pieces between the occurrences of 
critical subwords in the defining relators.
On seeing such a subword in $w'$, run the Turing machine $M$ which calculates
$f$ for $t_i$ steps on input $i$ and see if $M$ calculates $j$ in exactly
that number of steps.  
If so, we indeed have part of a conjugate of the relator $r_i$ or its inverse. 
Next, check that   the part of $w'$ containing the above subword has the correct form,
that is, generators are in the correct order and have the appropriate powers $j$ or $-j$.
If we expand the occurrence to a subword $s$ which is part of a relator 
$r \equiv st$ and $t$ is a product of  at most $7$ pieces, replace $s$ by
$t^{-1}$ and freely cyclically reduce the resulting word to obtain $w''$.
 Note that  the piece length of $w''$ is at least $12$ less than the piece
length of $w$.  Repeat until we do not find any critical  subwords.
The proof for the other  sets of relators we construct   will be  the same.
\end{proof}

  We review the  idea of \emph{generic-case complexity} from the  paper \cite{KMSS}.
Let   $\Sigma$ be a nonempty finite alphabet  and let
$S \subseteq \Sigma^*$. The \emph{density of $S$ at } $n$, written $\rho_n (S)$, is
the number of words in $S$ of length less than or equal to $n$ divided
by the number of all words of length less than or equal to $n$.
If   limit  $\rho(S) =lim_{n \to \infty} \rho_n(S) = 1$    
we say that $S$  is \emph{generic} in $\Sigma^*$.
   A particular decision problem $D$ on words over $\Sigma^*$ is said to 
\emph{ generically computable in time} $T(n)$ if there is a partial algorithm 
$\Phi$ for $D$ which answers correctly  on an input $w$ in time $T(|w|)$  or else does not give an answer 
and such that the domain of $\Phi$ is   generic in $\Sigma^*$.

   We point out that while worst-case complexity of the word or conjugacy
problems is independent of  a given  presentation  for a finitely generated group,  
this is not the case for generic-case complexity.  To show that a decision problem
having  a certain generic-case complexity is a property \emph{of the group $G$} one  needs  to  show 
that for \emph{every} finitely generated presentation of $G$ there is a partial  algorithm working 
in the given time bound.   For the groups which we construct, the conjugacy problem is generically
 linear time  by Theorem C  of \cite{KMSS} since our  groups have infinite abelianizations containing  $\mathbb{Z}^6$.
 From the form of the relators, after abelianization the $c$ and $d$ generators disappear, so they generate a
free abelian subgroup in the abelianization, and the abelianization is  independent of presentation.

\section{The groups for the Time Hierarchy Theorem}

    We now want to mirror the Time Hierarchy Theorem.  
We discuss a few details of the proof  following the presentation in Arora and Barak \cite{AB}.
They consider Turing machines with a special input tape, some number $k \ge 1$ of work tapes
and an output tape.  Such a machine is completely determined by a complete listing 
of its transition function, which  can easily be encoded by a string of $0$'s and $1$'s 
beginning with $1$,   and thus can be regarded as a binary number. 
Many strings are not valid codes of Turing machines.
We then express   strings coding   Turing machines as unary numbers and  
$M_x$ is the Turing machine coded by a unary $x$.
Of course,  there are infinitely many numbers coding machines with exactly the same behavior.  
(Just add an arbitrary number of  non-reachable states.)

 If   $f(n) \ge n$  is   a fully time constructable function then   
 there is a language $L_f  \subseteq \{1\}^*$ with $L \in TIME(f^2 (n))$ but such that  
$ L_f \notin TIME(f(n))$. 
It is a fact that  there is a universal Turing machine $\mathcal U$ which, on  input  $x$,   
simulates $M_x$ on input $x$ in time $|x|log(|x|)$.
Now  consider the Turing machine $\widehat{M}$ which, on unary input $x$, halts and outputs $1$ if $x$
does not code a Turing machine and otherwise 
 uses the universal machine $\mathcal{U}$ to simulate $M_x$ on  input $x$ for $f(x)^{1.5}$ steps.
If $M_x$ halts and outputs $1$   then  $\widehat{M}$ halts and outputs $0$.  
Otherwise $\widehat{M}$ halts and outputs $1$.  
Let $L_f$ be the set of unary inputs on which   $\widehat{M}$ halts and outputs $1$.
Then $L_f \subseteq $ DTIME($f^2(n)$) but  $L_f \nsubseteq $ DTIME($f(n)$).    
If   $L_f$   were in  DTIME($f(n)$), there would  a Turing machine $M$ which obtains the same output 
as $\widehat{M}$   in time $f(|n|)$.  For any constant $c$, there is an $n_0$ such that 
$n^2 > cn log(n)$ for all $n > n_0$.
Since there are infinitely many Turing machines with the same behavior as $M$, let $x$ be the code 
of such a machine with $x > n_0$.  Then $M_x$ would obtain the same result  as $\widehat{M}$ 
on input $x$ in time $f(|x|)$, a contradiction.

  We again use the free group  $F = \langle a_1, \dots, a_{20},b_1, \dots, b_{20}, c_1,c_2,c_3,d_1,d_2,d_3 \rangle$.   
Now  let $t_i$  be the time 
used by $\widehat{M}$ in  deciding  if $i \in L_f$.  We will use almost the same
set of relators as before.
The defining relators for the group $H$ for $f$ will be the symmetrized closure $R$ of the 
following set $\{r_j  :  j \in L_f  \}$ of relators where we now have

\[   r_j = {a_1}^j \dots {a_{20}}^j  {c_1}^j {d_1}^{t_j} {c_2}^j {d_2}^{t_j} {c_3}^j{d_3}^{t_j}  {b_{20}}^{-j} \dots {b_1}^{-j} {d_3}^{-t_j}  {c_3}^{-j} {d_2}^{-t_j} {c_2}^{-j} {d_1}^{-t_j} {c_1}^{-j} , j \in L_f   \]

     The same considerations as before  shows that this $R$ satisfies the piece condition $C(20)$
and that  there is a quadratic time algorithm which, given $w$, calculates a weakly cyclically reduced conjugate of $w$.  
Given the defining relators $R$, it will again be immediate from small cancellation theory that 

\[  {a_1}^j \dots {a_{20}}^j  \thicksim   {b_{1}}^{j} \dots {b_{20}}^{j}  
   \iff j \in L_f. \]
 
\noindent    and  that the conjugacy problem for  $H$ is not in $DTIME(f(n)$ but is in  $DTIME(f^2(n)$.

  We have

\begin{thm}  For every fully time constructable function $f$ with $f(n) \ge n^2$, there is a
computably presented group $H$ such that  for every fixed element $u_0$ of $H$, the problem 
of deciding if an arbitrary element is conjugate to $u_0$ is decidable in quadratic time 
while  the conjugacy problem of $H$ is  decidable in time $f^2(n)$ but is not decidable in 
time $f(n)$. The conjugacy problem of $H$ has linear time generic-case complexity.
\end{thm}

\section{The half-conjugacy problem}

   Our result on the half-conjugacy problem is the following.

\begin{thm} For every computable  function $f: \mathbb{N} \to \{0,1\}$ 
there is a computably presented group $P$  with solvable conjugacy problem for which    
the half-conjugacy problem  
 is solvable in quadratic time but the time complexity function of the conjugacy
problem satisfies $T(n) > f(n)$ for all $n \ge 1$. The group $P$ retains  the property that
all individual conjugacy problems are decidable in quadratic time and the conjugacy problem has
linear time generic-case complexity.
\end{thm}

    For the half-conjugacy problem, now let $f$ be any  computable  function $f: {\mathbb{N}}^{+} \to \{0,1\}$.
One can construct  a  Turing machine $\widehat{M}$ which, on unary input $i$, halts and outputs $1$ 
if $i$ does not code a Turing machine and otherwise 
 uses the universal machine $\mathcal{U}$ to simulate $M_i$ on  input $i$ for $f(40i)$ steps.
If $M_i$ halts and outputs $1$   then  $\widehat{M}$ halts and outputs $0$.  
Otherwise $\widehat{M}$ halts and outputs $1$.  
The machine $\widehat{M}$ computes a total function $g$. Let $t_i$ be the time used  by $\widehat{M}$ in computing $g(i)$.
Let $L_g$ be the set of unary inputs on which  outputs $1$.
Since we have diagonalized over $f$, $L_g$ is computable in time $g$ but not in time $f$.

   Again let  $F = \langle a_1, \dots, a_{20},b_1, \dots, b_{20}, c_1,c_2,c_3,d_1, d_2,d_3 \rangle$.  
The set $R$ of defining relators will  be $\{ r_i : i \ge 1\}$ where for $ i \in L_g$,

\[   r_i = {a_1}^i \dots {a_{20}}^i {c_{1}}^i {d_{1}}^{t_i} {c_{2}}^i {d_{2}}^{t_i} {c_{3}}^{i} {d_3}^{t_i}{b_{20}}^{-i} \dots {b_1}^{-i} {d_3}^{-t_i}{c_{3}}^{-i} {d_{2}}^{-t_i} {c_{2}}^{-i} {d_{1}}^{-t_i} {c_{1}}^{-i} ,   \]
  
while  for  $ i \notin L_g$,

\[   r_i = {a_1}^i \dots {a_{20}}^i  {c_{1}}^i {d_{1}}^{t_i} {c_{2}}^i {d_{2}}^{t_i} {c_{3}}^{i} {d_3}^{t_i}{b_{1}}^{i} \dots {b_{20}}^{i} {d_3}^{-t_i}{c_{3}}^{-i} {d_{2}}^{-t_i} {c_{2}}^{-i} {d_{1}}^{-t_i} {c_{1}}^{-i} .  \]

     The set   $R$ again satisfies the piece condition $C(20)$. 
Note that the relators show that for \emph{all} $i \ge 1$ the element
${a_1}^i \dots {a_{20}}^i$ is conjugate to either ${b_1}^i \dots {b_{20}}^i$ or to its inverse,
but which possibility holds depends on whether or not $i \in L_g$.
Furthermore, any  algorithm for the conjugacy problem  decides membership in $L_g$ and so
takes as much time as $g$.

 \section{The geometry of conjugacy diagrams}

     Although various results of small cancellation theory are
often  stated for a metric small cancellation condition, 
results about the geometry of the relevant diagrams depend only
on the  appropriate piece condition.  We have seen that given
an arbitrary element $w$ we can effectively find a weakly cyclically 
$R$-reduced conjugate $u$ of $w$ in quadratic time.
What we now need  is that given two weakly cyclically $R$-reduced words $u$
and $v$ which are not equal to the identity in $G$
and which are conjugate in $G$, the conjugacy diagram $\Delta$ for $u$ and $v$
satisfies the conclusion of Theorem 5.5 of Chapter V of Lyndon-Schupp.
 (\cite{LS}, page 257.)

\begin{theorem} Fix any  group $G$ among   the groups we have constructed.
Let $u$ and $v$ be two nontrivial weakly cyclically $R$-reduced words which are
conjugate in $G$ and let $\Delta$ be a reduced conjugacy diagram for 
$u$ and $v$ with outer boundary $\sigma$ and inner boundary $\tau$.  Then
every region of $\Delta$ has edges on both $\sigma$ and $\tau$,
has at most two interior edges and has  no interior vertices.
\end{theorem} 
 
 In short,  the theorem says that $\Delta$ ``looks like'' one of the diagrams in Figure 1.
 The essential difference between the two pictures is whether or not the inner and outer boundaries
have any vertices in common.

  
\begin{figure}

\scalebox{0.3}{\includegraphics{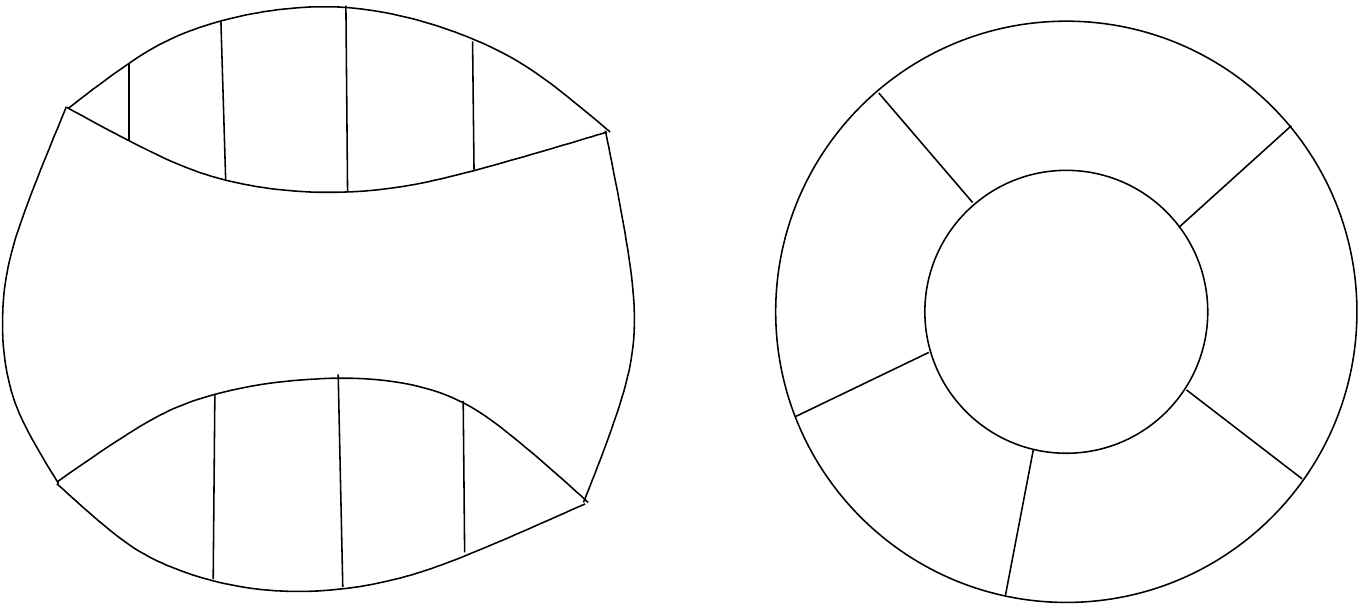}}

\caption{Possible conjugacy diagrams} \label{Fi:conjugacy}
\end{figure}

 That $u$ and $v$ are weakly cyclically $R$-reduced means that no cyclic 
permutation of either contains an element of $R$ with at most $7$ pieces
missing.  Since $R$ satisfies the condition $C(20)$, this means that
any region $D$ which intersected only $\sigma$ or $\tau$ and such
that the intersection is a consecutive part of the boundary would have 
 interior degree at least $12$, which will be  impossible by the counting formulas.
These formulas   depend only on the piece condition and thus the conclusion of 
Theorem 5.5 follows just as in the metric case.

 We give  the detailed argument for the groups $G$ with conjugacy problem of  desired c.e. degree.
  Given an arbitrary nonidentity element $u_0$ of $G$, fix a weakly cyclically reduced
 conjugate  $u$ of $u_0$.   For powers ${a_k}^l$  or ${b_p}^q$  which occur in $u$,
we need to know if $k$ or $q$ are in $A$ and, if so, what arguments of the function $f$
give those values.  We claim that this finite amount of
 information suffices  to decide conjugacy to $u$ and thus conjugacy to $u_0$.

   Given an element $v'$, we can  calculate a weakly cyclically $R$-reduced conjugate $v$ of $v'$
in quadratic time by the Reduction Lemma, Lemma  \ref{l: reduction}.
If  $v \thicksim  u$ in the free group $F$, they are conjugate in $G$. If not,
but   $v \thicksim  u$ in $G$, there is 
 a minimal conjugacy diagram $\Delta$ for $u$ and $v$ containing at least one region.
Let $\sigma$,  labelled by $u$, be the outer boundary of  $\Delta$, and let $\tau$,
labelled by $v$, be the inner boundary of    $\Delta$.  
The structure theorem shows that, since the piece length of the intersection
of the boundary of any  region $D$ with  either the outer or inner boundary of $\Delta$
is a most $12$, the intersection of the boundary of $D$ with  both  $\sigma$ and $\tau$ 
 must contain several occurrences of generators to the same power $j$.
Since we are just considering conjugacy to $u$,  the free group length $C = |u|$ of $u$ is
a constant in  this algorithm.

   In a conjugacy diagram $\Delta$ for  $u$ and $v$, say that an edge $e \in \sigma$ and an edge $f \in \tau$ are ``opposite each other''
if one of the following conditions holds:
\begin{enumerate}

\item the two edges are on both boundaries  and coincide;

\item  the edges  form the beginning of an island in that the edge preceding them is on both boundaries and $f^{-1}e$ are successive edges 
on the boundary of a region of $\Delta$;

\item $f^{-1}he$ is a successive part of the boundary of a region $D$ and $h$ is the label on an interior  edge separating $D$ 
from another region.

\end{enumerate}

 The main point is that if we choose a letter $y$ in $u$   and a letter $z$ in $v$ and suppose that they are the 
labels on edges which are opposite   each other in $\Delta$, then is  \emph{only one way} to fill in the rest of the conjugacy diagram 
and we can calculate whether or not a valid conjugacy diagram with this initial condition exists in linear time.
If we succeed in constructing a valid conjugacy diagram for one of the $3C|v|$ possible initial conditions, then
$u$ and $v$ are conjugate, and if we cannot construct a conjugacy diagram then they are not conjugate.

 It is also clear from the structure of conjugacy diagrams that

\[  {a_1}^j \dots {a_{20}}^j   \thicksim  {b_{1}}^{j} \dots {b_{20}}^{j}  
   \iff  j \in A. \]

so a solution to the conjugacy problem for $G$ decides membership in $A$.
On the other hand, given an oracle for the set $A$ we can calculate weakly cyclically $R$-reduced conjugates for any pair 
of elements and then apply the method above to  decide  conjugacy and  the conjugacy problem for $G$ is Turing equivalent
to deciding membership in $A$.

   For the  groups $H$ for the Time Hierarchy  Theorem it is clear that

\[  {a_1}^i \dots {a_{20}}^i  \thicksim  {b_{1}}^{i} \dots {b_{20}}^{i}  
   \iff i \in L_f \]

so the conjugacy problem cannot be calculated in time $f$.

   The conjugating part of the defining relators has been chosen so that the only way that $a_{20}{c_1}^id_1$  or $d_1{c_2}^id_2$
 or $d_2{c_2}^id_3$ or their inverses can be subwords of  the label on an interior edge in a reduced diagram is if 
they are matched against the corresponding generators of the other conjugating part of the same relator.  And since the boundary
labels of a conjugacy diagram are freely reduced, the entire conjugating parts must then be exactly matched and appear as the label
on the interior edge.  If all interior edges are so labelled, the diagram shows that a power of some  ${a_1}^i \dots {a_{20}}^i$
is conjugacy to the same  power of    ${b_{1}}^{i} \dots {b_{20}}^{i} $.   

   If it is the case that for some region of the conjugacy diagram a   conjugating part of the relator is not completely
matched against its inverse, then  some  ${c_l}^i{d_l}^{t_i}$ occurs  on one of the boundaries and the padding allows us
to calculate in linear time if this is indeed a correct part of a relator.  In this case we can again see if one can construct 
a valid conjugacy diagram in quadratic time. 
So  the algorithm in time $f^2(n)$ for membership in $L_f$ solves the conjugacy problem for $H$.

   The remarks for the groups $H$ apply exactly to the groups $P$ for the half-conjugacy problem. The only hard case is
when a power of $ {a_1}^i \dots {a_{20}}^i$ is  conjugate to the corresponding power of $ {b_{1}}^{i} \dots {b_{20}}^{i}$
or to its inverse.  The relators force one of the two possibilities to hold but deciding which one requires deciding membership
in $L_g$.  So  the conjugacy problem
for these groups is solvable in time $g$ but not in time $f$.

\section{A group with $\mathcal{NP}$-complete conjugacy problem}

   We now want to construct a computably presented group $G$ with $\NP$-complete conjugacy
problem while keeping the constraint that all individual conjugacy problems are decidable
in quadratic time and that the conjugacy problem for the given presentation
is strongly generically quadratic time.   
The previous results on imitating the Time-Hierarchy Theorem and
the half-conjugacy problem depended on using the  free group unary notation since 
this notation gives enough padding in the conjugating parts of the relators to check the
correctness of relators.
Of course, elements of a free group also have a unique \emph{normal form with exponents},
where we write powers of generators as the name of the generator with a decimal exponent.
A \emph{syllable} is such a power of a generator and we require that adjacent syllables are
powers of distinct generators.  For example, one element of the free group $\langle a,b,c \rangle$\
is $w = a^{25} b^{-17}a^{-33}c^{3}$.   The \emph{length} of a normal form with exponents is
 the total number of symbols in the normal form. Thus $|w| = 13$ for the  example just given.
Basic decision problems in free groups mainly retain their
polynomial-time decidability in this notation.  For example, Gurevich and Schupp \cite{GS} show
that the uniform membership problem for finitely generated subgroups of a free group remains in polynomial 
time when  elements are written in exponent normal form.  We need to use exponent normal form 
in order to have a  coding of the satisfiability problem for Boolean expressions where the coded  
length is proportional to the length of the standard coding of such expressions.

   The problem  \emph{3-SAT} is  the satisfiability problem for 
Boolean expressions which are  conjunctions of clauses, each of which contains exactly
three  literals.  A literal is the symbol $x$ with positive decimal subscript, representing 
a variable, or its negation.  For example, the expression 
$$( x_1 \lor x_3 \lor \lnot x_7)  \land ( \lnot x_4 \lor x_7 \lor x_{11})  \land (x_1 \lor  x_7 \lor \lnot x_9)  \land ( \lnot x_3 \lor  x_4 \lor  x_9)$$
is an instance of $3$-SAT.  A basic result of complexity theory is  that $3$-SAT is $\mathcal{NP}$-complete. 
Variables may be repeated in a clause but we assume that a clause  does not both a variable and
its negation.

   We will represent the clause $(x_{i} \lor x_j \lor x_k)$ as $a^i b^j c^k$ with the exponent negative 
if the variable is negated.  We represent a conjunction of clauses by the concatenation of the 
representatives of the clauses.
Thus we represent the example of $3$-SAT given above by
$$ a^1 b^3 c^{-7} a^{-4} b^7 c^{11} a^1 b^{7} c^{-9}  a^{-3} b^4 c^9 $$    
So  the length of our basic coding is even shorter than the standard coding of instances of $3$-SAT.

  In order to represent $3$-SAT we need to code all its instances.  To do this we put a ``short-lex''
 well-ordering on $\mathbb{Z}^3$ as follows.  The \emph{index} of a triple $(z_1,z_2, z_3)$ is 
$|z_1| + |z_2| + |z_3|$, the sum of the absolute values of the $z_i$.  We order triples first by index, 
and within the same index  lexicographically but with positive values preceding negative ones.  
Note that we  consider  only triples  which do not  contain $0$ since subscripts of variables are positive.

We define an enumeration $\mathcal{E} = \{ \eta_i \}$ of all instances of $3$-SAT as follows.  We enumerate all pairs $(m,n)$
of positive integers in the usual way.  When a pair $(m,n)$ is enumerated, we then enumerate all
instances of $3$-SAT where there are at most $m$ clauses and the maximum index of any clause is at most $n$
and the instance \emph{has not previously been enumerated}. The clauses in an instance are concatenated
in the short-lex order defined above and we order instances lexicograpically.

 A language $L \in \NP$ is  characterized by the fact that for every instance which is in $L$,
there is a short \emph{certificate},  given which  one can verify in polynomial time that the
instance is indeed in $L$.  For $3$-SAT this certificate is an assignment of truth values showing that
the instance is actually  satisfiable.  For an instance of $3$-SAT which is satisfiable we choose the
the  first satisfying assignment in the usual truth-table order.

  We code this satisfying assignment using generators $x,y,z$ in the following way. 
For our running example given above, the first line of the truth-table assigning  all variables
the value false satisfies the instance.  So we code this assignment as
$$ x^{-1} y^{-3} z^{-7}  x^{-4} y^{-7} z^{-11} x^{-1} y^{-7} z^{-9} x^{-3} y^{-4} z^{-9}$$
 
  The absolute values of the exponents again represent the variables and the sign of
the exponent is negative if the variable is assigned the value false and positive if the
variable is assigned the value. 
  
  The defining relators we use will be words in exponential normal form in the free group $F$ on generators
$$ a_1,b_1,c_1,....,a_{20}, b_{20},c_{20}, d_1,e_1,f_1,..., d_{20},e_{20}, f_{20}, u_1, v_1, w_1,..., u_3,v_3,w_3 ,x_1, y_1,z_1, x_2, 
y_2,z_2  $$

  If  $\eta_j$ is   a satisfiable instance in the enumeration $\mathcal{E}$ defined above, 
let $\alpha_{\eta_j, l}$ denote the coding of this instance on the generators $a_l,b_l,c_l$ for $1 \le l \le 20$
and  let $\beta_{\eta_j,l}$ denote the coding of this instance on the generators $d_l,e_l,f_j$ 
for $1 \le l \le 21$.  Let $\gamma_{\eta_j,l}$ represent the coding of  this instance 
on the generators $u_l, v_l,w_l$ for $1  \le l \le 3$.  Finally, let $\delta_{\eta_j,l}$ represent the
coding of the satisfying truth assignment for this instance on the generators $x_l, y_l,z_l,  l = 1,2$.

  The basic defining relators for our group $G$ is the set $R = \{r_{\eta_j}\}$ where $r_{\eta_j}$ is
$$ \alpha_{\eta_j,1}...\alpha_{\eta_j,20} (\gamma_{\eta_j,1} \delta_{\eta_j,1}\gamma_{\eta_j,2} \delta_{\eta_j,2}\gamma_{\eta_j,3})
 {\beta_{\eta_j,20}}^{-1}...{\beta_{\eta_j,1}}^{-1}  ({\gamma_{\eta_j,3}}^{-1}  {\delta_{\eta_j,2}}^{-1} {\gamma_{\eta_j,2}}^{-1}  {\delta_{\eta_j,2}}^{-1} {\gamma_{\eta_j,1}}^{-1})  $$  
   
 and where $\eta_j$ ranges over all \emph{satisfiable} instances in the enumeration $\mathcal{E}$.
Note that although the Greek letters now represent long words on the given generators, taking them to correspond
to the Roman letters of the previous groups shows that the general form of the defining relators is essentially the same  as before.  
 
    We  now  need to use small cancellation theory over free products, which is essentially like small
cancellation theory over free groups.   For technical details we again refer to  Lyndon-Schupp\cite{LS},
but we review some basic definitions.   
A free group $F$ with a specified free basis  can be viewed as the free product of the infinite cyclic groups generated by 
the specified generators. The free product normal form is given by the normal form with exponents but  
the \emph{free product length} $|u|$ of a normal form  is just the number of syllables in $u$.
We now view $F$ as this free product.

     If $ u = ys_1$ and $v = s_2 z$ are free product normal forms with the last syllable of $u$ and the first syllable of 
$v$  in the same factor, then there is \emph{cancellation} in the product $uv$ if $s_2 = {s_1}^{-1}$
and \emph{consolidation} in the product $uv$ if   $s_2 \ne {s_1}^{-1}$.  An element $u = y_1...y_n$ is
\emph{weakly cyclically reduced} if $|u| \le 1$ or   $y_n \ne {y_1}^{-1}$.  A set $R \subset F$ is
\emph{symmetrized} if every element of $R$ is weakly cyclically reduced and if $r \in R$ then every weakly cyclically
reduced conjugate of $r$ and $r^{-1}$ is also in $R$.

    An element $w$ has \emph{semi-reduced form} $uv$ if there is no cancellation in the product $uv$.
Note that consolidation is allowed. An element $p$ is a \emph{piece relative to R} if
$R$ contains \emph{distinct} elements $r_1$ and $r_2$ with semi-reduced forms $r_1 = py_1$ and  $r_2 = p^{-1}y_2$.
The metric small cancellation condition is now 

\emph{Condition} $C'(\lambda)$: If $r \in R$ has semi-reduced form $r = py$ where $p$ is a piece, then
$|p| < \lambda |r|$.  Also, every element of $R$ has length greater than $1/\lambda$.

   The set $R$ of defining relators we now use is the symmetrized closure of the set $R$ defined above,
that is,  all weakly cyclically reduced conjugates of elements of $R$ and their inverses.
Since there can be only  one basic relator corresponding to a particular instance of $3$-SAT,
it is easy to see that the maximum length of a piece relative to $R$ is at most  the length of
the conjugating part of the relator.  
Thus our set $R$ of relators satisfies $C'(1/9)$.  

   Given this small cancellation condition the geometry of conjugacy diagrams is the same as
in the case of quotients of free groups (not viewed as free products).  
Thus  in the quotient group $F/N$ where $N$ is the normal closure of $R$ we have

$$ \alpha_{\eta_j,1} ... \alpha_{\eta_j,20}      \thicksim    \beta_{\eta_j,1} ... \beta_{\eta_j,20}  $$

if and only if the instance $\eta_j$ is satisfiable. 
 It follows exactly as in our previous discussion 
 that powers of variants of the above are the only hard instances.
Thus each individual conjugacy problem is decidable in quadratic time and the conjugacy problem as
a whole is  generically linear  time. Thus  the conjugacy problem of $G$ is $\NP$-complete
and the  other requirements are met.

\section{Groups with algorithmically finite conjugation}

\medskip \noindent
  We have  shown that  one can bound the complexity of all individual conjugacy problems while
making the global conjugacy problem arbitrarily complex.  In Miller's famous examples of residually finite, 
finitely presented groups $G$ with undecidable conjugacy problem, there is  always  some element $q$ for which $ICP(q)$ 
is undecidable (Lemma 4 on page 27 in \cite{Miller}).  
However,  \cite{BMR3} shows that     the  individual  conjugacy problems
 $ICP(w)$  in Miller's group are  solvable in polynomial time for for all $w$ from an exponentially generic subset of $G$. 
This leads one to ask about the opposite phenomenon. 
 
 \begin{question}
 Are there  recursively presented groups $G$ with solvable word problem such that if the individual conjugacy problems are
  decidable on a computably enumerable subset $Y \subseteq G$ then $Y$ is negligible, or indeed exponentially negligible?
 \end{question}

Although  Theorem \ref{th:Turing} shows that there is no general  effective  way  to build a 
uniform decision algorithm for the  conjugacy problem from solutions of the individual conjugacy problems,
 the following general lemma holds.

\begin{lemma}  Let $G = \langle X; R \rangle$  be a recursively presented group. If $W = \{ w_1,...,w_n\}$ is a finite set of
pairwise nonconjugate elements of $G$ then there is a partial algorithm $\Phi$ which decides the conjugacy problem
on the union  $Z = \bigcup_{i =1,...,n} {w_i}^G$ of the conjugacy classes of the $w_i \in W$. 
\end{lemma}

\begin{proof}
The partial algorithm $\Phi$ works as follows.  Since $G$ is recursively presented, when given elements $u,v \in G$, we can
 begin enumerating in parallel all words equal in $G$ to conjugates of the $w_i \in W$.
If  $u$ and $v$ are in $Z$, they will both eventually  be enumerated in this process.  If they are enumerated  
as conjugates of the same $w_i$ they represent conjugate elements of $G$.  If they are enumerated as conjugates
of distinct $w_i$ and $w_j$ they are not conjugate in $G$.
\end{proof}

\begin{corollary} A recursively presented group with only finitely many conjugacy classes has solvable conjugacy problem.
\end{corollary}

    We say that a recursively presented group $G$   has {\em algorithmically finite conjugation}  if $G$ has infinitely
many conjugacy classes and every infinite c.e. 
set of elements of $G$ must contain two elements which are conjugate.  We note that this condition implies that if $Y$ 
is a c.e. set  of elements of $G$ for which there is a partial algorithm $\Phi$ solving the conjugacy problem for 
elements of $Y$ then  $Y$ must have elements from only   finitely many conjugacy classes  and we are in the situation 
of the above lemma.  
Given a partial algorithm $\Phi$ deciding conjugacy for a set  $Y$ containing infinitely many pairwise non-conjugate elements, 
 we could  computably  enumerate an infinite set $S$ of pairwise non-conjugate elements of $G$ as follows.  Let $s_0$ be the
first element in the enumeration of $Y$. Now use $\Phi$ and the enumeration  $Y$ until we find an element $s_1$ not conjugate
to $s_0$.  Continue this process, finding at the $n$-th stage, an element $s_n$ not conjugate to any of 
$s_0,...,s_{n-1}$.  

  Thus the conjugacy problem is as bad as possible  in a group with algorithmically finite conjugation.

     Recall that a group $G$ generated by a finite set $X$ is termed {\em algorithmically finite} \cite{MO, KhoussM} 
if every infinite computably enumerable subset of $F(X)$ has two distinct words which define equal elements in $G$. 
In other words one can computably enumerate only a finite set of words in $F(X)$ which define pair-wise distinct elements of $G$. 
Infinite, recursively presented, algorithmically
finite groups are also called Dehn Monsters and have been shown to exist \cite{MO}. Indeed, there are even residually finite  
Dehn Monsters\cite{KhoussM,Klyachko}.
We next observe that any Dehn Monster has algorithmically finite conjugacy.

\begin{theorem} \label{th:monsters}
Let $G$ be an infinite, recursively presented, algorithmically finite  group generated by a finite set $X$.  Then:
\begin{itemize}
\item [1)] $G$ has infinitely many conjugacy classes;
\item  [2)] $G$ has algorithmically finite conjugation;  
 \end{itemize}
\end{theorem}

\begin{proof}
Since $G = \langle X;R \rangle$ has unsolvable word problem it has unsolvable
conjugacy problem and hence must have infinitely  many conjugacy classes by 
the above corollary.  That $G$ has algorithmically finite conjugacy is immediate since
any infinite  c.e. set must have two distinct words which are equal in $G$ and
thus certainly conjugate.
\end{proof}

    In Dehn Monsters  one can  solve the conjugacy problem only  on finite unions of conjugacy classes, but there is still
a question about  the asymptotic density of  single conjugacy classes.  There are non-amenable finitely generated groups 
with finitely many conjugacy classes \cite{Osin}, so not all conjugacy  classes in such groups are negligible. Also,
the question  of whether or not there are finitely generated, residually finite, non-amenable groups with 
only finitely many conjugacy classes seems to be open.

\bigskip

\end{document}